\documentclass[12pt,letterpaper]{article}
\usepackage{graphicx}
\usepackage{amsmath}
\usepackage{amsfonts}
\usepackage{amsthm}
\usepackage{amssymb}
\usepackage{color}
\def\N{\mathbb{N}}
\def\Z{\mathbb{Z}}
\def\R{\mathbb{R}}
\usepackage{verbatim} 
\usepackage{hyperref}

\usepackage{enumerate}

\newtheorem{prop}{Proposition}
\newtheorem{cor}[prop]{Corollary}
\newtheorem{thm}[prop]{Theorem}
\newtheorem{rem}[prop]{Remark}
\newtheorem{lemma}[prop]{Lemma}

\newcommand{\be}{\begin{equation}}
\newcommand{\ee}{\end{equation}}
\newcommand{\bea}{\begin{eqnarray}}
\newcommand{\eea}{\end{eqnarray}}
\newcommand{\beas}{\begin{eqnarray*}}
\newcommand{\eeas}{\end{eqnarray*}}

\newcommand\ol{\overline}
\newcommand\ul{\underline}
\newcommand\dsprod{\displaystyle\prod}
\newcommand\dssum{\displaystyle\sum}
\newcommand\vp{\varphi}

\newcommand\UU{\mathbb{U}}

\begin{document}

\title{Sequences of Integers with Missing Quotients and 
Dense Points Without Neighbors}

\author{Tanya Khovanova\\MIT \and Sergei Konyagin\\Steklov Mathematical 
Institute}
\maketitle

\begin{abstract}
Let $A$ be a pre-defined set of rational numbers. 
We say a set of natural numbers $S$ is an $A$-quotient-free set if no ratio of two elements in $S$ belongs to $A$. We find the maximal asymptotic density and the maximal upper asymptotic density of $A$-quotient-free sets when $A$ belongs to a particular class.

It is known that in the case $A = \{p,\ q\}$, where $p$, $q$ are coprime integers greater than one, the latest problem is reduced to evaluation of the largest number of lattice non-adjacent points in a triangle whose legs 
lie on coordinate axis. We prove that this number is achieved by 
choosing points of the same color in the checkerboard coloring.
\end{abstract}


\section{Introduction}

\subsection{Definitions and notation}\label{sec:def}

Suppose $A$ is a finite set of positive rational numbers. A set $S$ of positive
integers is called \emph{an $A$-quotient-free set}
if no ratio of two elements in $S$ belongs to $A$. In this paper we will study 
density properties of $A$-quotient-free sets.

For $c\in \R$ we denote $[c] = \{k\in\N:\,k\le c\}$.
In particular, for $n\in\N$ we have $[n] = \{1,\ 2,\ \ldots,\ n\}$.

For any set $S$ of positive integers, let $\ol{\delta} (S)$ be the upper 
asymptotic density of $S$:
$$\ol{\delta} (S) = \limsup_{X\to\infty} \frac{|S \cap [X]|}{X}.$$
Similarly, we define the lower asymptotic density $\ul{\delta} (S)$ as: 
$$\ul{\delta} (S) = \liminf_{X\to\infty} \frac{|S \cap [X]|}{X}.$$
In particular, the upper asymptotic density and the lower asymptotic density 
of finite sets is 0. 
If $\ol{\delta} (S) = \ul{\delta} (S)$, then we say that $S$ has the 
asymptotic density $\delta (S)$, where 
$\delta (S) = \ol{\delta} (S) = \ul{\delta} (S)$.

Next, we define the upper logarithmic density of a set $S\subset\N$:
$$\ol{\delta}_{\log} (S) = \limsup_{X\to\infty} 
\frac{\sum_{k\in S, k\le X}1/k}{\ln  X}$$
and the lower logarithmic density of a set $S\subset\N$:
$$\ul{\delta}_{\log} (S) = \liminf_{X\to\infty} 
\frac{\sum_{k\in S, k\le X}1/k}{\ln  X}.$$
If $\ol{\delta}_{\log} (S) = \ul{\delta}_{\log} (S)$, then we say that $S$ has 
the logarithmic density $\delta_{\log}(S)$, where 
$\delta_{\log}(S) = \ol{\delta}_{\log}(S) = \ul{\delta}_{\log}(S)$.

It is known that
\be\label{logasym}
\ul{\delta} (S)\le \ul{\delta}_{\log} (S) \le \ol{\delta}_{\log} (S)
\le \ol{\delta} (S)
\ee
(see \cite{Ten}, Part~III, Theorem~2).
Therefore, if $\delta(S)$ exists then $\delta_{\log}(S)$ also exists
and $\delta_{\log}(S)=\delta(S)$.

We are interested in describing the highest density achieved by $A$-quotient 
free sets. We define the following measures of density limitations: 
$\ol{\rho}(A)$, $\ul{\rho}(A)$, $\rho (A)$ (see \cite{Chen}).
For 
$$\ol{\rho}(A) = \sup_S \ol{\delta} (S) \text{\  and \ } 
\ul{\rho}(A) = \sup_S \ul{\delta} (S),$$
the supremum is taken over all $A$-quotient-free sets $S$. For
$$\rho(A) = \sup_S \delta (S)$$
the supremum is taken over all $A$-quotient-free sets $S$ for which 
$\delta (S)$ exists. 

Similarly, we define $\ol{\rho}_{\log}(A)$, $\ul{\rho}_{\log}(A)$, $\rho_{\log}(A)$:

$$\ol{\rho}_{\log}(A) = \sup_S \ol{\delta}_{\log}(S) \text{\  and \ } 
\ul{\rho}_{\log}(A) = \sup_S \ul{\delta}_{\log}(S)$$
and 
$$\rho_{\log}(A) = \sup_S \delta_{\log}(S).$$

It is easy to see that 
\be\label{aslimitineq}
\ol{\rho} (A) \geq \ul{\rho} (A) \geq \rho (A)
\ee
and
\be\label{loglimitineq}
\ol{\rho}_{\log}(A) \geq \ul{\rho}_{\log}(A) \geq \rho_{\log}(A).
\ee
Also, by (\ref{logasym}),
\be\label{asloglimitcompare}
\ol{\rho} (A) \geq \ol{\rho}_{\log}(A),\quad
\ul{\rho} (A) \leq \ul{\rho}_{\log}(A),\quad
\rho(A) \leq \rho_{\log}(A).
\ee

\subsection{Known results}\label{sec:known}

The following known results are of our interest (see \cite{Chen}).

Let $A = \{a_1, a_2, \ldots, a_r\}\subset\N$, where 
$1 < a_1 < a_2 < \cdots < a_r$ and $a_i$ are pairwise coprime. Let 
$M = M(a_1,\ldots,a_r) = \{m_1 < m_2 < \cdots\}$ be an ordered set of integers 
of the form $a_1^{u_1}a_2^{u_2}\cdots a_r^{u_r}$, where 
$u_i \geq 0$. The function $f(t) = f(A,t)$ denotes the maximal 
cardinality of $A$-quotient-free subsets of $\{m_1, m_2, \ldots, m_t\}$.

\begin{thm}\label{thm1} We have
$$\rho(A) \geq\frac{1}{2}\left(1+\dsprod_{i=1}^r \frac{a_i-1}{a_i+1}\right).$$
\end{thm}

\begin{thm}\label{thm2} We have
$$\ol{\rho}(A) = \dsprod_{i=1}^r \left(1-\frac{1}{a_i}\right) 
\dssum_{t=1}^\infty f(t)\left(\frac{1}{m_t} - \frac{1}{m_{t+1}}\right).$$
\end{thm}

Although we know the exact value of $\ol{\rho}(A)$ from Theorem~\ref{thm2},
it is not easy to apply this theorem directly since in general we do not know
a simple algorithm to evaluate $f(t)$ for large $t$. To describe a more 
explicit lower estimate for $\ol{\rho}(A)$ obtained in \cite{Chen}, we need 
some more notation. 

Let
$$m_i=a_1^{u_{i1}}\dots a_r^{u_{ir}},\quad i=1,2,\dots.$$
For $j=0,1$ we define
$$A_j(t)=A_j(a_1,\dots,a_r,t)=\{m_i:\,
u_{i1}+\dots+u_{ir}\equiv j\pmod{2},\,i=1,\dots,t\}.$$
Let
$$\ol\sigma(A)=\dsprod_{i=1}^r\dssum_{t=1}^\infty
\max\left(|A_0(a_1,\dots,a_r,t)|,|A_1(a_1,\dots,a_r,t)|\right)
\left(\frac 1{m_t} - \frac 1{m_{t+1}}\right).$$

The more explicit estimate obtained in \cite{Chen}
as a corollary from Theorem~\ref{thm2} is

\begin{cor}\label{cor1} We have
$$\ol{\rho}(A) \ge \ol{\sigma}(A).$$
\end{cor}

\subsection{Statement of new results}\label{sec:newres}

\begin{thm}\label{thm3} For any finite set $A$ of positive rational numbers 
we have
$$\rho(A)=\ul\rho(A)=\rho_{\log}(A)=\ul\rho_{\log}(A)
=\ol\rho_{\log}(A).$$
Moreover, there exists an $A$-quotient-free set $S$ such that
$\delta(S)=\rho(A)$.
\end{thm}

Theorem~\ref{thm3} gives a negative answer to Question~2 from \cite{Chen}, which asked whether there are any sets $A$ such that $\ol{\rho} (A) = \ul{\rho} (A) \neq \rho (A)$.

We evaluate $\rho(A)$ in terms of the solution of an extremal problem
on the set of subsets of $\Z_+^s$ for some $s$. 

If $A$ is a set described in subsection~\ref{sec:known}, then we can 
find an explicit value of $\rho(A)$.

\begin{thm}\label{thm4} Let $A = \{a_1, a_2, \ldots, a_r\}\subset\N$, where 
$1 < a_1 < a_2 < \cdots < a_r$ and $a_i$ are pairwise coprime. Then
$$\rho(A) = \frac{1}{2}\left(1+\dsprod_{i=1}^r \frac{a_i-1}{a_i+1}\right).$$
\end{thm}

So, we show the sharpness of the estimate in Theorem~\ref{thm1}.

If, moreover, $r=2$, then we find a convenient expression for $\ol\rho(A)$ below. 

\begin{thm}\label{thhm5} Let $A=\{p,q\}\subset\N$, where $1<p<q$ and 
$\gcd(p,q)=1$. Then
$$\ol{\rho}(A) = \ol{\sigma}(A).$$
\end{thm}

So, we show the sharpness of the estimate in Corollary~\ref{cor1}.

The proof of Theorem~\ref{thhm5} is based on a combinatorial result which 
might have an independent interest. Before stating our related results, we 
will give some comments on reduction of the problem of finding the densest
$A$-quotient-free sets to a combinatorial problem in the plane. 

We are interested in finding the largest $A$-quotient-free subset in $[N]$.
We say that two integers $i$ and $j$ belong to the same equivalency class if 
$i/j = p^x q^y$, where $x$ and $y$ are integers. In other words, there exists 
$z$ not divisible by $p$ and $q$, such that $i = p^{x_1}q^{y_1}z$ and 
$j = p^{x_2}q^{y_2}z$, where $x_i$ and $y_i$ are non-negative integers.
The ratio of two numbers from different equivalency classes cannot belong to 
the set $A$. Thus, to find the largest $A$-quotient-free subset in $[N]$ we 
need to find the largest such subset in every class. That is, for every $z$ we 
need to find the largest $A$-quotient-free subset among the numbers of the 
form $p^xq^y z$ such that they do not exceed $N$:  $p^xq^y \leq N/z = n$. 
That means we can ignore $z$ and study $A$-quotient-free subset of the numbers 
of the form $p^xq^y$ in the range $[n]$. Equivalently, we can arrange numbers 
of the form $p^xq^y$ into an increasing sequence and study the largest 
$A$-quotient-free subsets among the first $t$ elements of this sequence.

The above arguments were used in \cite{Chen} for the proof of 
Theorem~\ref{thm2}.

So, we are interested in finding the largest $A$-quotient-free subset 
of numbers of the form $p^xq^y$, bounded by $n$: $p^xq^y \leq n$, or, 
equivalently, $(\ln p) x + (\ln q) y \leq \ln n$. Our goal translated into the 
new formulation is to find the maximum number of points with integral 
coefficients, lattice points, in the given triangle 
defined by the inequalities $x \geq 0$, $y \geq 0$ and $ax + by \leq c$, 
where $a=\ln p>0$, $b=\ln q >0$ and $c=\ln n>0$ such that no two points 
are horizontally or vertically adjacent. We will keep $a$ and $b$ fixed and 
change $c$. We denote the triangle as $\triangle_c$, and the maximum number of 
points is $f(t)$, where $t$ is the total number of lattice points in the 
triangle $\triangle_c$.

Let us color the lattice points on the plane in the checkerboard order. 
Lattice points with an even sum of coordinates we color in white and with an 
odd sum of coordinates in black. See Figure~\ref{fig:checkerboard}.

\begin{figure}[htb!]
\centering
\includegraphics[scale=0.6]{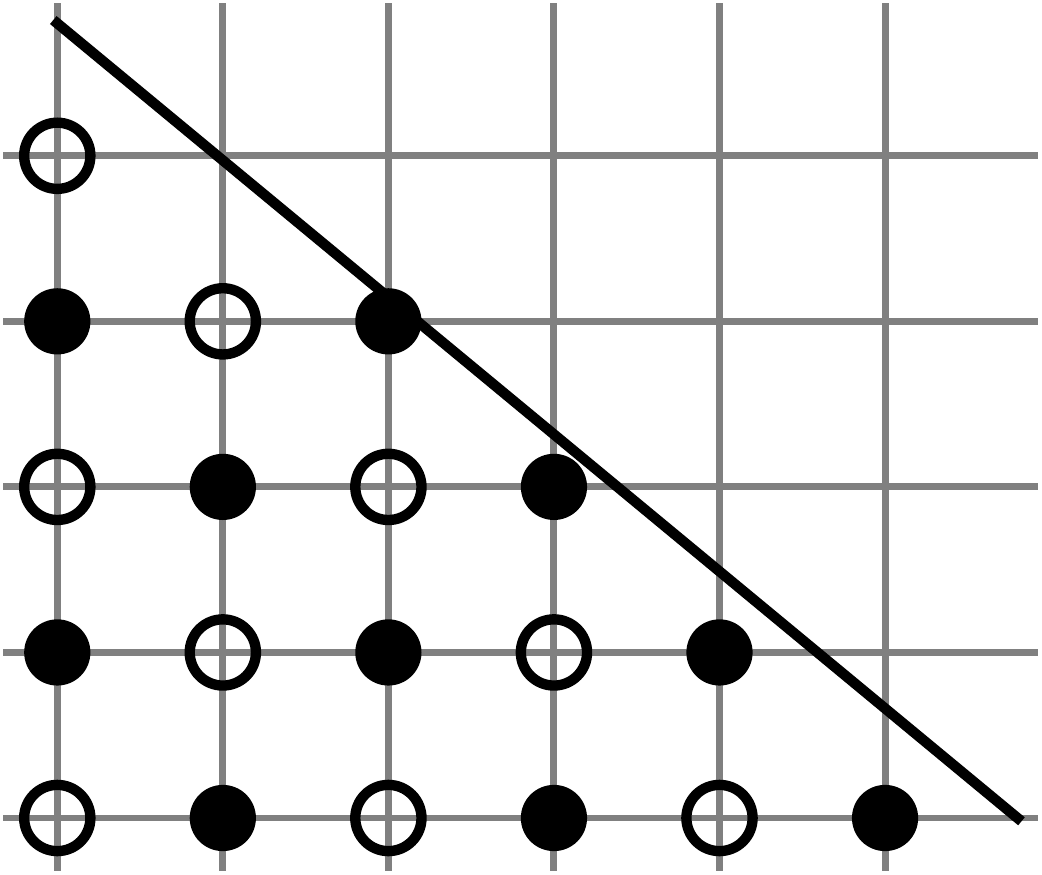}
\caption{Checkerboard coloring.}
\label{fig:checkerboard}
\end{figure}

Points of the same color are never adjacent. So the points of the same color 
are good candidates for our largest set. In Figure~\ref{fig:checkerboard} 
we can count 9 white points 
and 10 black points inside the triangle. 
Thus, the black points correspond to a larger $A$-quotient-free subset. The 
theorem below shows that we cannot find a larger set than the points of 
the same color.

\begin{thm}\label{thm6}
Given a line $ax + by \leq c$, where $a > 0$, $b > 0$ and $c > 0$, the 
maximum number of points with non-negative integer coordinates $\{x,\ y\}$, 
that satisfy the condition $ax + by \leq c$ is the maximum of two numbers: the 
number of black points satisfying the condition or the number of white points.
\end{thm}

If $t$ is the total number of lattice points in the triangle $\triangle_c$, 
then $A_0(p,q,t)$ ($A_1(p,q,t)$) is the number of white (black) points in the 
triangle $\triangle_c$.

\begin{cor}\label{cor2} If $1<p<q$ and $\gcd(p,q)=1$, then
$$f(\{p,q\},t)=\max\left(|A_0(p,q,t)|,|A_1(p,q,t)|\right).$$
\end{cor}

Corollary~\ref{cor2} is Conjecture~1 from \cite{Chen}, and, thus, is proved.

Theorem~\ref{thhm5} immediately follows from Theorem~\ref{thm2} and 
Corollary~\ref{cor2}.

Also, the above arguments give the following corollary of 
Theorem~\ref{thm6}.

\begin{cor} If $1<p<q$ and $\gcd(p,q)=1$, $N\in\N$, then
the maximal cardinality of an $A$-quotient-free subset of $[N]$ is
$$\sum_{\substack{n\le N,\\p\nmid n,q\nmid n}}
\max\left(|A_0(p,q,t(N/n))|,|A_1(p,q,t(N/n))|\right),$$
where $t(u)$ for $u\ge1$ is defined by $m_{t(u)} \le u < m_{t(u)+1}$.
\end{cor}

It was proved in \cite{Chen} that there are infinitely many sets 
$A\subset\N$ such that $\ol\rho(A)>\rho(A)$. We have strengthened that 
result as follows.

\begin{thm}\label{thm7} If $r\ge2$, $A = \{a_1, a_2, \ldots, a_r\}\subset\N$, 
where $1 < a_1 < a_2 < \cdots < a_r$ and $a_i$ are pairwise coprime, then
$\ol\rho(A)>\rho(A)$. 
\end{thm}

The proof of Theorem~\ref{thm7} is based on the following result.
Suppose we have a plane $\alpha_1x_1 + \alpha_2x_2 + \cdots + \alpha_rx_r = c$ 
in an $r$-dimensional space and 
$\alpha_1 < \alpha_2 < \cdots < \alpha_r$. We color the space in the checkerboard 
order, so that the origin is white.

\begin{thm}\label{thm8}
If $\alpha_1/\alpha_2$ is irrational, then there exists $c > 0$ such that the number of black points in the simplex 
formed by the plane and the coordinate hyper-planes is greater than the number of 
white points.
\end{thm}

\section{The Proof of Theorem~\ref{thm3} and Theorem~\ref{thm4}}
\label{sec:dens}

\subsection{An upper estimate}\label{sec:upper}

For a set $A=\{a_1,\dots,a_r\}$ of positive rational numbers 
we can find a set $B=\{b_1,\dots,b_s\}$ of pairwise coprime
positive integers so that each element $a_i\in A$ has a representation
$$a_i=\dsprod_{j=1}^s b_j^{u_{ij}},\quad u_{ij}\in\Z\,
(i\in [r],\,j\in [s]).$$
For example, one can write all elements $a_i$ as irreducible fractions and
take $B$ as the set of all prime divisors of all numerators and denominators.

We will use the following lemmas.

\begin{lemma}\label{lemma1} Let $M(B)$ be the set of integers
of the form $b_1^{x_1}\dots b_s^{x_s}$, $x_j\in\Z_+\,(1\le j\le s)$
and $N(B)$ be the set of integers $n\ge1$ with $b_j\not| n$
$(1\le j\le s)$. Then
\begin{enumerate}[(i)]
\item every positive integer $k$ can be represented uniquely as $k= mn$, 
$m\in M(B)$ and $n\in N(B)$;
\item if $(mn)/(m'n')=b_1^{u_1}\dots b_s^{u_s}$, $u_j\in\Z\,(j=1,\dots,s)$, 
$m,m'\in M(B)$ and $n,n'\in N(B)$, then $n=n'$.
\end{enumerate}
\end{lemma}

Part (i) of Lemma~\ref{lemma1} is part (i) of Lemma 1 from \cite{Chen}. 
Part (ii) easily follows from part (i).

Let
$$\vp(B) = \dsprod_{j=1}^s\left(1-\frac1{b_j}\right).$$

\begin{lemma}\label{lemma2} For any $X\in\N$ we have
$$||N(B)\cap[X]| - \vp(B)X|< 2^s.$$
\end{lemma}

\begin{proof} Since $b_1,\dots,b_s$ are pairwise coprime,
for any $1 \leq j_1<\dots<j_{\nu} \leq s$ the number of $n\le X$, divisible by 
all $b_{j_1},\dots,b_{j_{\nu}}$, is 
$$\left \lfloor \frac X{b_{j_1}\dots b_{j_\nu}}\right \rfloor.$$
Lemma follows form the inclusion-exclusion formula

$$|N(B)\cap[X]|=\dssum_{\nu=0}^s
(-1)^{\nu}\sum_{j_1<\dots<j_{\nu}}
\left \lfloor \frac X{b_{j_1}\dots b_{j_\nu}}\right \rfloor$$
and the inequalities
$$0 \le \frac X{b_{j_1}\dots b_{j_\nu}}
- \left \lfloor \frac X{b_{j_1}\dots b_{j_\nu}}\right \rfloor <1.$$
\end{proof}

\begin{lemma}\label{lemma3} We have
$$\dssum_{n\le X, n\in N(B)}\frac1n = \vp(B)\ln X +O(1).$$
\end{lemma}

(Here and throughout the following, the implicit constants in $O$
depend only on $B$.)

\begin{proof} Define $F(x)=|N(B)\cap[x]|$. By partial summation,
we have
$$\dssum_{n\le X, n\in N(B)}\frac1n
=\dssum_{x\le X-1}\frac{F(x)}{x(x+1)}+\frac{F(X)}X,$$
and the result follows from Lemma~\ref{lemma2}.
\end{proof}

Denote
$$\ol u_i=\{u_{i_1},\dots,u_{i_s}\}\in\Z^s,\quad(i=1,\dots,r)$$
and 
$$\UU=\{\ol u_1,\dots,\ol u_r\}.$$
We say that a set $E\subset\Z_+^s$ is an $\UU$-difference free set
if no two elements of $E$ differ by an element of $\UU$.
For any set $S\subset\N$ and $n\in N(B)$ we can define
the set
$$E(S,n)=\left\{(u_1,\dots,u_s)\in\Z_+^s:\,
n\dsprod_{j=1}^s b_j^{u_j}\in S\right\}.$$
It is easy to set that the set $S$ is an $A$-quotient-free set 
if and only if for any $n$ the set $E(S,n)$ is an $\UU$-difference free 
set.
Define the magnitude
$$\gamma(A,B)=\sup\dssum_{(u_1,\dots,u_s)\in E}
\dsprod_{j=1}^s b_j^{-u_j}$$
where the supremum is taken over $\UU$-difference free sets $E$.
We will prove that all density limitations involved in Theorem~\ref{thm3}
are equal to $\vp(B)\gamma(A,B)$ (and, therefore, $\vp(B)\gamma(A,B)$ does not 
depend on the choice of $B$).

Now we are ready to prove an upper estimate for $\ol\rho_{\log}(A)$.

\begin{lemma}\label{upest} We have $\ol\rho_{\log}(A) \le \vp(B)\gamma(A,B)$.
\end{lemma}

\begin{proof} Let $S$ be an $A$-quotient-free set. 
Take a large number $X$. We have
$$S=\bigcup_{n\le X, n\in N(B)} S_{X,n},$$
where
$$S_{X,n}=\{x\in S: x\le X, x=mn, m\in M(B)\}.$$
Since the set $E(S,n)$ is an $\UU$-difference free set, we get
$$\dssum_{x\in S_{X,n}}\frac1x \le \frac{\gamma(A,B)}n.$$
By Lemma~\ref{lemma3},
\begin{align*}
&\dssum_{x\in X, x\le X}\frac1x=\dssum_{n\le X, n\in N(B)} 
\dssum_{x\in S_{X,n}}\frac1x\le\gamma(A,B)
\dssum_{n\le X, n\in N(B)}\frac1n\\
&= \vp(B)\gamma(A,B) \ln X  +O(1),
\end{align*}
and the assertion follows.
\end{proof}

\subsection{A lower estimate}\label{sec:lower}

We use notation from the previous subsection.
First, we show that the supremum in the definition of $\gamma(A,B)$
is attained.

\begin{lemma}\label{existence} There exist an $\UU$-difference free set
$\hat E\subset\Z_+^s$ such that 
\be\label{attain}
\dssum_{(u_1,\dots,u_s)\in \hat E}
\dsprod_{j=1}^s b_j^{-u_j} = \gamma(A,B).
\ee
\end{lemma}

\begin{proof}
Let $\Omega$ be the product of topological spaces $\{0,1\}$:
$$\Omega=\{\{\omega_{\ol u}\}:\,\ol u\in\Z_+^s,\,\omega_{\ol u}\in\{0,1\}\}.$$
Every element $\omega\in\Omega$ defines a subset $E(\omega)\subset\Z_+^s$ as
$$E(\omega)=\{\ol u\in\Z^s:\,\omega_{\ol u}=1\}.$$
thus, we have a one-to-one correspondence between $\Omega$ and the set of 
subsets of $\Z_+^s$.

Define the set $W\subset\Z_+^s\times\Z_+^s$ as
$$W=\{(\ol u,\ol v):\,\exists i\in\{1,\dots,r\}\,\ol v-\ol u=\ol u_i \}.$$
Next, for $w=(\ol u,\ol v)\in W$ we define the subset
$\Lambda_w\subset\Omega$ as
$$\Lambda_w=\{\omega: \omega_{\ol u}\omega_{\ol v}=0\}.$$
Let
$$\Omega_0=\bigcap_{w\in W}\Lambda_w.$$
We notice that the set $E(\omega)$ is an $\UU$-difference free set
if and only if $\omega\in\Omega_0$.

By Tikhonov theorem (see, for example, \cite{Ku}, Appendice, Th\'eor\`eme~7), 
the space $\Omega$ is compact. Since all sets 
$\Lambda_w$ are closed subsets of $\Omega$, we conclude that $\Omega_0$ is 
also compact. We consider the following function $F:\Omega\to\R$: 
$$F(\omega)=\dssum_{\substack{(u_1,\dots,u_s)\in\Z_+^s,\\
\omega_{(u_1,\dots,u_s)}=1}}\dsprod_{j=1}^s b_j^{-u_j}.$$ 
It is easy to see that $F$ is continuous and that
$$\sup_{\omega\in\Omega_0}F(\omega) = \gamma(A,B).$$
Due to compactness of $\Omega_0$, the supremum is attained at some 
$\hat\omega\in\Omega_0$. For the set $\hat E= E(\hat\omega)$ 
equality~(\ref{attain}) holds. The proof of the lemma is complete.
\end{proof}

\begin{lemma}\label{lowest} There exists an $A$-quotient-free set
$S$ such that $\delta(S)=\vp(B)\gamma(A,B)$.
\end{lemma}

\begin{proof} Due to Lemma~\ref{existence}, we take an $\UU$-difference free 
set $\hat E$ satisfying (\ref{attain}). We denote
$$M=M(B),\quad N=N(B).$$  
Let
$$S=\{mn: n\in N, m=b_1^{u_1}\dots b_s^{u_s}\in M, 
(u_1,\dots,u_s)\in\hat E\}.$$
For any $n\in N$ we have $E(S,n)=\hat E$. Hence,
$S$ an $A$-quotient-free set. It suffices to check that 
$\delta(S)=\vp(B)\gamma(A,B)$.

Let $M = \{m_1 < m_2 < \cdots\}$. Observe that $m_1=1$. 
Denote by $f_0(t)$ the number of $m\in\{m_1,\dots,m_t\}$
such that $m=b_1^{u_1}\dots b_s^{u_s}$, $(u_1,\dots,u_s)\in\hat E$.
Let $X\in\N$, $t\in\N$ and $S_t$ be the set of $mn\in S$,
$m\in M$, $n\in N$, $X/m_{t+1} < n \le X/m_t$. We have
\be\label{expdiff}
|S\cap[X]|=\dssum_{t=1}^\infty |S_t\cap[X]|
=\dssum_{t=1}^\infty f_0(t)\left(|N\cap[X/m_t]|-|N\cap[X/m_{t+1}]|\right).
\ee
(Notice that for $m_t>X$ all the summands are equal to $0$.)
By partial summation,
\begin{align*}
&|S\cap[X]|=\dssum_{t=1}^\infty |N\cap[X/m_t]| (f_0(t)-f_0(t-1))
=\sum_{f_0(t) > f_0(t-1)}|N\cap[X/m_t]|\\
&=\sum_{(u_1,\dots,u_s)\in\hat E}\left|N\cap\left[
X\dsprod_{j=1}^s b_j^{-u_j}\right]\right|.
\end{align*}
Denote
$$E_X=\{(u_1,\dots,u_s)\in\hat E:\,u_1+\dots+u_s\le\log_2 X\}.$$
Observe that if $(u_1,\dots,u_s)\in \hat E\setminus E_X$
then  
$$\dsprod_{j=1}^s b_j^{-u_j} \le \dsprod_{j=1}^s 2^{-u_j}<1/X.$$
Therefore,
$$|S\cap[X]| = \sum_{(u_1,\dots,u_s)\in E_X}\left|N\cap\left[
X\dsprod_{j=1}^s b_j^{-u_j}\right]\right|.$$
For $X\ge2$ we have $|E_X|=O((\ln X)^s)$. Applying Lemma~\ref{lemma2} we get
$$|S\cap[X]| = \sum_{(u_1,\dots,u_s)\in E_X}
\vp(B)X\dsprod_{j=1}^s b_j^{-u_j}+O\left((\ln X)^s\right).$$
Since
$$\dssum_{(u_1,\dots,u_s)\in E_X}\dsprod_{j=1}^s b_j^{-u_j}
= \dssum_{(u_1,\dots,u_s)\in \hat E}\dsprod_{j=1}^s b_j^{-u_j}+o(1)
=\gamma(A,B)+o(1)$$
we get
\be\label{rho1}
|S\cap[X]| = \vp(B)\gamma(A,B)X+o(X)
\ee
as desired.
\end{proof}

\begin{rem} Similarly to those arguments we have used to prove
(\ref{rho1}), one can deduce from (\ref{expdiff}) another expression for
$\delta(S)$, namely,
$$\delta(S) = \dssum_{t=1}^\infty f_0(t)\left(\frac 1{m_t}-\frac 1{m_{t+1}}
\right).$$
Therefore,
\be\label{twoexp}
\vp(B)\gamma(A,B) = \dssum_{t=1}^\infty f_0(t)
\left(\frac 1{m_t}-\frac 1{m_{t+1}}\right).
\ee
\end{rem}

\begin{proof}[Proof of Theorem~\ref{thm3}] 
We know from Lemma~\ref{upest} and Lemma~\ref{lowest} that
$$\ol\rho_{\log}(A) \le \vp(B)\gamma(A,B) \le \rho(A).$$
Therefore, by (\ref{aslimitineq}), (\ref{loglimitineq}) and 
(\ref{asloglimitcompare}), we get
$$\rho(A)= \ul\rho(A)= \rho_{\log}(A)=\ul\rho_{\log}(A)
=\ol\rho_{\log}(A)=\vp(B)\gamma(A,B)$$
as desired.
\end{proof}

Also, we see from (\ref{twoexp}) that
\be\label{another}
\rho(A) =\vp(B)\gamma(A,B)= 
\dssum_{t=1}^\infty f_0(t)\left(\frac 1{m_t}-\frac 1{m_{t+1}}\right).
\ee

\subsection{The proof of Theorem~\ref{thm4}}\label{proofthm4}

In the case when $A = \{a_1, a_2, \ldots, a_r\}\subset\N$,
$1 < a_1 < a_2 < \cdots < a_r$ and $a_i$ are pairwise coprime,
we take $B=A$, $s=r$ and $b_i=a_i$ ($i=1,\dots,r$). The condition that
a set $E\subset\Z_+^r$ is an $\UU$-difference free set
means that no two points of $E$ are adjacent
with respect to any line parallel to a coordinate line
(see a related discussion for $r=2$ in subsection~\ref{sec:newres}).
Denote $\gamma(A)=\gamma(A,A)$. Take
$$\hat E=\{(u_1,\dots,u_r)\in\Z_+^r:\,u_1+\dots+u_r\equiv0\pmod{2}\}.$$
In other words, if we color $\Z_+^r$ in the checkerboard order, so that the 
origin is white, then $\hat E$ is the set of white points. We prove that
$\gamma(A)$ is attained for $E=\hat E$.

\begin{lemma}\label{white} We have
$$\dssum_{(u_1,\dots,u_r)\in \hat E}
\dsprod_{j=1}^r a_j^{-u_j} = \gamma(A).$$
\end{lemma}

\begin{proof} We use induction on $r$. For $r=0$ there is nothing to prove.
Let $r>0$. Assuming that the assertion is true for $r-1$ we will
prove it for $r$.

We introduce some notation. For $s\le r$ (actually, the cases 
$s=r$ and $s=r-1$ will be interesting for us) and $E\subset\Z_+^s$,
let
$$F_s(E)=\dssum_{(u_1,\dots,u_s)\in E}\dsprod_{j=1}^s a_j^{-u_j}.$$
Denote
$$\sigma=F_{r-1}(\Z_+^{r-1}).$$
For $E\subset \Z_+^s$ and $\ol u\in E$ we consider that 
$\ol u=(u_1,\dots,u_s)$. For $E\subset\Z_+^r$ and $k\in\Z_+$, let
$$E_k=\{\ol u\in E: u_r=k\},\quad
E^k=\{\ol u\in\Z_+^{r-1}:\,(u_1,\dots,u_{r-1},k)\in E_k\}.$$
Observe that for any even $k$ the set $\hat E^k$ is the set $\hat E'$ of 
white points in $\Z_+^{r-1}$. 

Since for any $\UU$-difference free set $E\subset\Z_+^r$ and any
$l\in\Z_+$ the sets $E^{2l}$ and $E^{2l+1}$ are disjoint. Thus,

\begin{align*}
&\dssum_{\ol u\in E_{2l}\cup E_{2l+1}}\dsprod_{j=1}^r b_j^{-u_j}
\le a_r^{-2l}\left(F_{r-1}(E^{2l})+a_r^{-1}(\sigma - F_{r-1}(E^{2l}))
\right)\\
&= a_r^{-2l-1}\sigma + a_r^{-2l}\left(1-a_r^{-1}\right)F_{r-1}(E^{2l}).
\end{align*}

Taking the sum over $l$ we get
\be\label{indstepineq}
F(E)\le\dssum_{l\in\Z_+}a_r^{-2l-1}\sigma
+\dssum_{l\in\Z_+}a_r^{-2l}\left(1-a_r^{-1}\right)F_{r-1}(E^{2l}).
\ee
For $E=\hat E$ inequality (\ref{indstepineq}) is an equality:
\be\label{indstepeq}
F(\hat E)=\dssum_{l\in\Z_+}a_r^{-2l-1}\sigma
+\dssum_{l\in\Z_+}a_r^{-2l}\left(1-a_r^{-1}\right)F_{r-1}(\hat E').
\ee
All sets $E^{2l}$ are $\UU$-difference free sets. By induction supposition, 
we have
$$F_{r-1}(E^{2l})\le F_{r-1}(\hat E').$$
Thus, we get from (\ref{indstepineq}) and (\ref{indstepeq})  
$$F(E)\le F(\hat E).$$
This concludes the proof of the induction step and completes the proof 
of the lemma.
\end{proof}

\begin{proof}[Proof of Theorem~\ref{thm4}] 
By Lemma~\ref{white} and equality~(\ref{another}) we have
$$\rho(A) =\vp(A)\gamma(A).$$ 
It was shown in the proof of Theorem~4 in \cite{Chen}
that
$$\vp(A)\gamma(A) = \frac{1}{2}
\left(1+\dsprod_{i=1}^r \frac{a_i-1}{a_i+1}\right).$$
This completes the proof of the theorem.
\end{proof}

\section{The Proof of Theorem~\ref{thm6}}\label{sec:cardinality}

Recall that we consider the triangle $\triangle_c$ that is defined as an area on the $\{x,\ y\}$ plane: $x \geq 0$, $y \geq 0$, $ax + by \leq c$, where $a$, $b$ and $c$ are positive.  We want to find the largest set of lattice points belonging to the triangle $\triangle_c$ that do not have vertical or horizontal neighbors in the set. 

Suppose we found such a set $S$ with the largest number of points. The basic idea is to move the points in $S$ around inside the triangle without changing the total number of points and without creating new adjacencies in such a way that at the end all the points are the same color. 

Let us take a diagonal $x+y=n$. Let us call the part of our triangle $\triangle_c$ that is on or below the diagonal \emph{the area of interest}. We will move all the points in the optimal set $S$ that lie inside the area of interest so that they become the same color while staying inside the area of interest. We will do this using induction by $n$. We will start with $n=0$ and the induction ends as soon as the whole triangle $\triangle_c$ lies below the diagonal $x+y=n$.

Let us provide the basis for induction. We start with the smallest $n$ such that there exist points in $S$ on the diagonal $x+y=n$. All these points are the same color.

Now assume that for a number $n$ we already moved points in $S$ so that all the points inside the area of interest are of the same color. New $S$ is also an optimal set, so we will refer to it as $S$. Our triangle $\triangle_c$ and the diagonal $x+y=n+1$ will generate a new area of interest. We will prove that all the points in $S$ that lie in the new area of interest can be moved to make them the same color and without creating adjacenies.

All the points on the diagonal $x+y=n+1$ are of the same color. Suppose that $n+1$ is even and that color is white. The case when the color is black is similar, and we will not discuss it. If all the points in the old area of interest are also white we do not have to do anything. Thus, suppose that those points are black.

Now consider several cases. 

\begin{enumerate}
\item The diagonal $x+y=n+1$ intersects the triangle side lying on the line $ax+by=c$.
\item The diagonal $x+y=n+1$ does not intersect the triangle side lying on the line $ax+by=c$. Not all the points with non-negative integer coordinates on the diagonal belong to the set $S$.
\item The diagonal $x+y=n+1$ does not intersect the triangle side lying on the line $ax+by=c$. All the points with non-negative integer coordinates on the diagonal belong to the set $S$.
\end{enumerate}

The first case. The diagonal $x+y=n+1$ intersects the line $ax+by=c$. Suppose the part of the diagonal that is inside the triangle is to the right of the intersection point. In particular that means that the lattice points of the diagonal has an $x$ coordinate greater than zero. Hence, their left neighbors belong to the triangle.

The points on the diagonal that belong to the set $S$ do not have neighbors in the set. In particular that means that each white point on the diagonal that belongs to the set $S$ has a left neighbor that is not in $S$. Now we move all the points in the set $S$ that are on the diagonal $x+y=n+1$ one step to the left. Our white points become black points. No new adjacencies inside the new area of interests are created because all of the points there are black. No new adjacencies between new points and the points outside the new area of intersect are created because the diagonal points moved away from the outside points. See Figure~\ref{fig:firstcase}.

\begin{figure}[htb]
\centering
\includegraphics[scale=0.6]{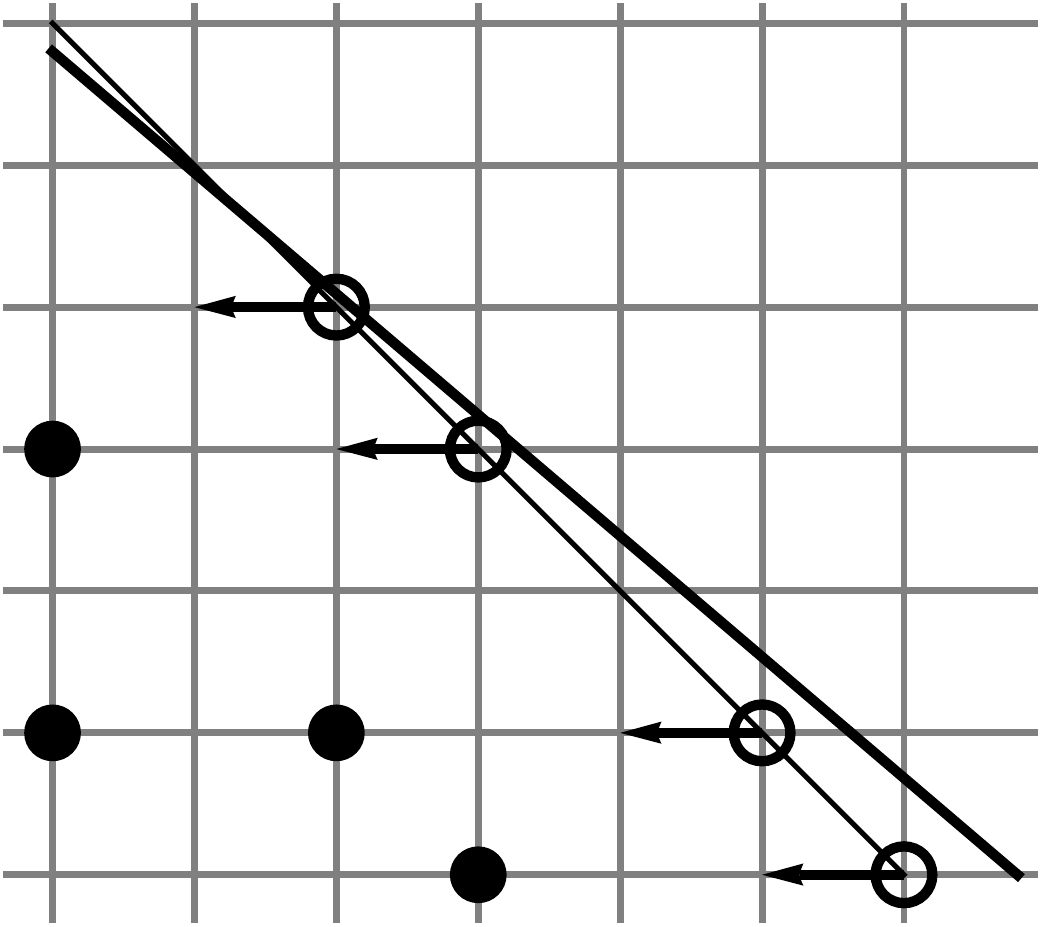}
\caption{The First Case.}
\label{fig:firstcase}
\end{figure}

The second case. The diagonal $x+y=n+1$ does not intersect the line $ax+by=c$, and not all the lattice points with non-negative integers coordinates on the diagonal belong to the set $S$. That means there is a point $P$ with non-negative integer coordinates on the diagonal that does not belong to the set. Now we move all the white points on the diagonal and in $S$ that lie to the left of $P$ one step down. Similarly we move one step to the left all the white points on the diagonal and in $S$ that lie to the right of $P$, see Figure~\ref{fig:secondcase}. The absence of point $P$ in $S$ guarantees that our movements will not collide: points that are moved to the left cannot get into the same position as points that are moved down. Similarly to the first case no new adjacencies are created.

\begin{figure}[htb!]
\centering
\includegraphics[scale=0.6]{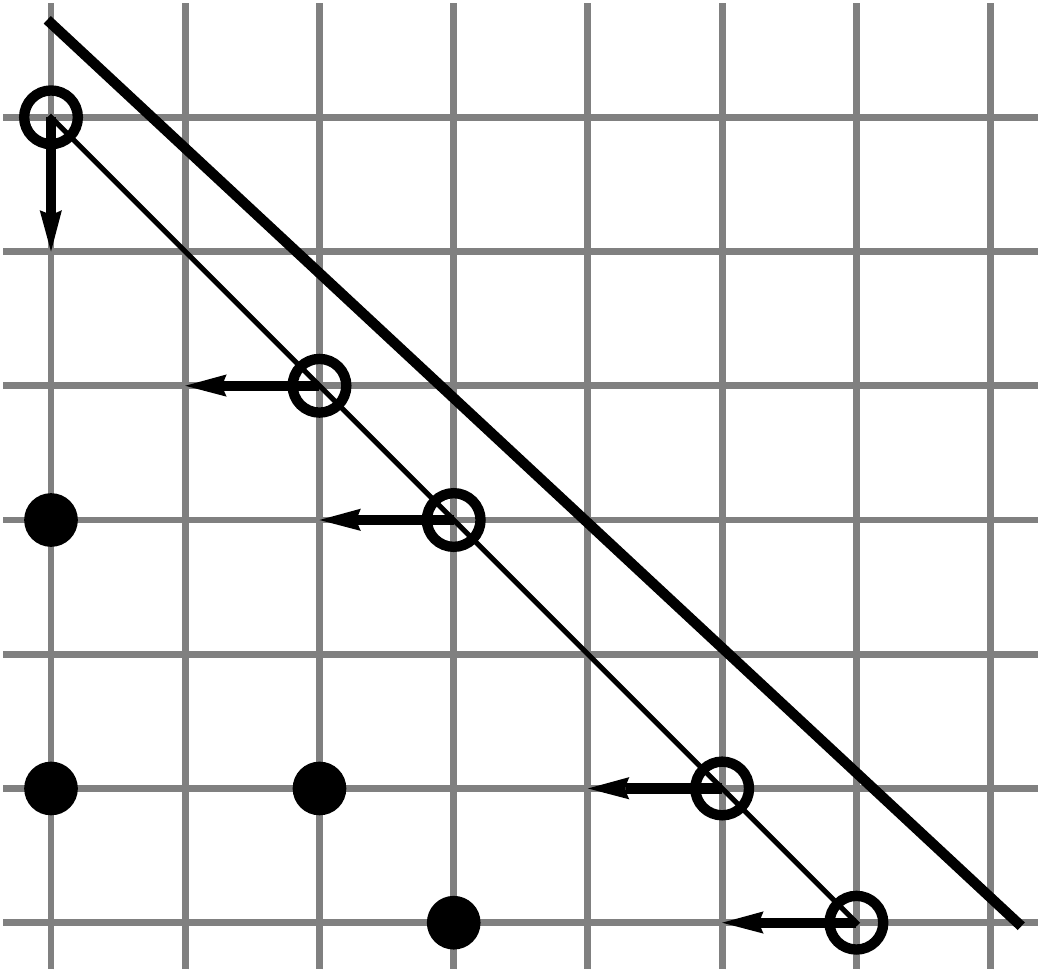}
\caption{The Second Case.}
\label{fig:secondcase}
\end{figure}

The third case. The diagonal $x+y=n+1$ does not intersect the line $ax+by=c$, and all the points with non-negative integer coordinates on the diagonal belong to the set $S$. We cannot move the white diagonal points down or left, because there are more white points on the diagonal $x+y=n+1$ than empty black spots on the diagonal $x+y=n$. But on the plus side, the whole diagonal $x+y=n+1$ belongs to the set $S$ and, therefore, all its neighbors do not. Hence, all the points on the diagonal $x+y=n$ do not belong to the set $S$. Therefore, all the black points in the old area of interest lie, in fact, below the diagonal $x+y=n$. We can move all of them up (or to the right for that matter) one step. They will all become white, but none of the points will become adjacent to the diagonal $x+y=n+1$. Thus, no new adjacencies are created.

Thus, the induction step is proven, and the theorem follows.

\begin{rem}
One might wonder what if we replace our triangle that have sides on the axis 
with any triangle. Will the theorem still be true? Here we show a counter-example. The triangle in Figure~\ref{fig:other} 
contains two lattice points $\{1,\ 0\}$ and $\{0,\ 2\}$. The points are not 
adjacent, so both of them can be included in the maximum set. But they are of 
different color.
\end{rem}

\begin{figure}[htb!]
\centering
\includegraphics[scale=0.3]{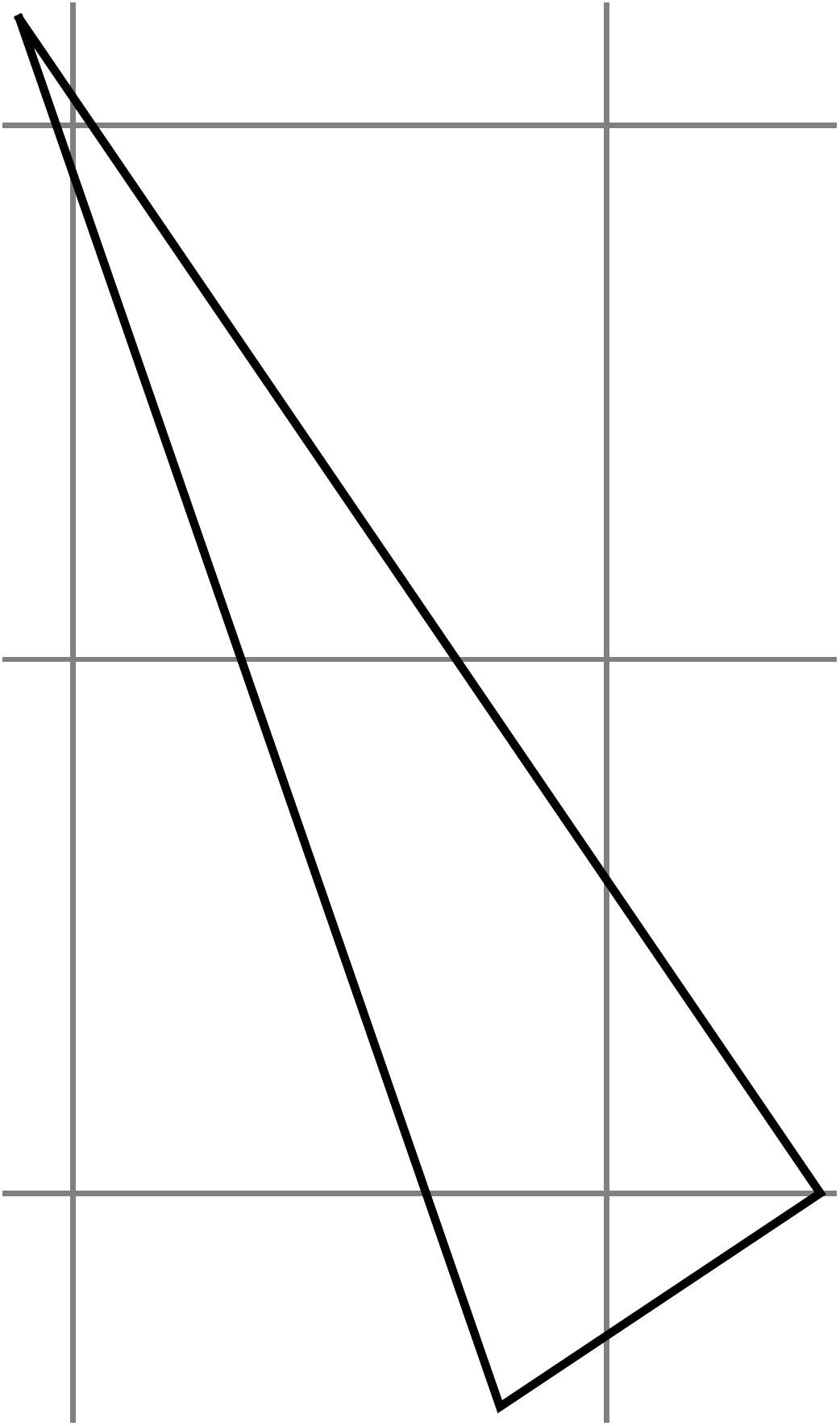}
\caption{A Counter-Example.}
\label{fig:other}
\end{figure}

\section{The Proof of Theorem~\ref{thm8} and Theorem~\ref{thm7}}
\label{sec:densineq}

\subsection{White Points Versus Black Points}\label{sec:slopes}

In this section we want to discuss a question of which color wins in the triangle $\triangle_c$ depending on the parameters of the line $ax+by=c$. First, if the line is very close to the origin and above it, we have only one point inside the triangle and it is white. This statement can be extended into many dimensions.

It is tempting to think that we can always choose white points. We later show that it is not true. On the other hand, if in every class we choose white points we will get an $A$-quotient free set with a simple description. The numbers of the form $p^iq^j z$, where $z$ is relatively prime with $p$ and $q$ and $i+j$ is divisible by 2, generate an $A$-quotient-free set.

Let us move to many dimensions. Consider a simplex $\triangle_c$ formed by a plane $\alpha_1x_1 + \alpha_2x_2 + \cdots + \alpha_rx_r = c$ and coordinate planes $x_i=0$ in an $r$-dimensional space. Let us assume that $\alpha_1 < \alpha_2 < \cdots < \alpha_r$. 

In Theorem~\ref{thm8} we claim that if $\alpha_1/\alpha_2$ is irrational, then there exists $c > 0$ such that the number of black points in the simplex $\triangle_c$ is greater than the number of 
white points.

\begin{proof}
First consider the plane with $c = \alpha_2$. If a lattice point in the simplex $\triangle_c$ has $x_1 > 0$, then all other coordinates of this point must be zero. Moreover, only one point in the simplex is such that $x_i > 0$ for $i > 1$, namely, the point $\{0,1,0,0,\ldots,0\}$. And this point is black.

The number of the points on the $x_1$ axis is $\lfloor \alpha_2/\alpha_1 \rfloor$. If this number is even we have more black points in the simplex $\triangle_c$ than white points. If this number is odd, then the number of black points is the same as the number of white points. Now we will move the plane by increasing $c$ slightly, so that one more point on the $x_1$ axis will belong to the simplex. 

The new $c$ which we denote as $c_1$ is $c_1 = \alpha_1 \lfloor \alpha_2/\alpha_1 \rfloor +\alpha_1$. The number of points on the $x_1$ axis in the simplex becomes even. Notice that more lattice points inside the simplex might appear, but all other points will have all but one coordinates equal to zero, and the non-zero coordinate equal to one. Thus, all extra points are black too. Thus, we found a plane which cuts out more black points than white points.
\end{proof}

We are not much interested in what happens when $\alpha_2/\alpha_1$ is 
rational. Indeed, the case when $A = \{a_1, a_2, \ldots, a_r\}\subset\N$, where 
$1 < a_1 < a_2 < \cdots < a_r$ and $a_i$ are pairwise coprime is translated into a geometric problem with $\alpha_i = \ln a_i$. Thus, the ratio $\alpha_2/\alpha_1$ is 
always irrational.

But we want to mention that the theorem can not be extended 
into cases when $\alpha_2/\alpha_1$ is rational. Consider a 2-dimensional 
space and suppose that $\alpha_1= 1$ and $\alpha_2=2$. The line $\alpha_1 x_1 + \alpha_2 x_2 =c$ passes through lattice points on the plane only when $c \in \Z$. If 
$c = 4k > 0$, $k\in\Z_+$, the number of white points is larger than the number of 
black points by one. If $c= 4k+1$, or $c = 4k + 2$, or $c = 4k + 3$, then the 
number of differently colored points is the same.

\subsection{The proof of Theorem~\ref{thm7}}\label{proofthm7}

As in subsection~\ref{proofthm4}, we take $B=A$. By (\ref{another}), we have
$$\rho(A) = 
\dssum_{t=1}^\infty f_0(t)\left(\frac 1{m_t}-\frac 1{m_{t+1}}\right).$$
By Theorem~\ref{thm4} we can write, in the terminology of 
subsection~\ref{sec:known} that $f_0(t)=A_0(t)$. Therefore,
$$\rho(A) = 
\dssum_{t=1}^\infty A_0(t)\left(\frac 1{m_t}-\frac 1{m_{t+1}}\right).$$
By Corollary~\ref{cor1},
$$\ol{\rho}(A) \ge 
\dssum_{t=1}^\infty \max(A_0(t), A_1(t))
\left(\frac 1{m_t}-\frac 1{m_{t+1}}\right).$$
So, equality $\rho(A)=\ol{\rho}(A)$ can hold only if
$A_0(t)\ge A_1(t)$ for all $t$. But this is excluded by Theorem~\ref{thm8}.
So, we have $\ol{\rho}(A) > \rho(A)$ as desired.

\section{Acknowledgments}

This work was conceived, carried out and completed while the authors were visiting the Institute for Advanced Study in Princeton, NJ. We are grateful to the IAS for its hospitality and excelent working conditions.


\end{document}